\documentclass[a4paper,11pt,leqno]{amsart}

\usepackage{lmodern}
\usepackage{amssymb,bm}
\usepackage[hidelinks]{hyperref}

\numberwithin{equation}{section}

\newtheorem{Theorem}{Theorem}[section]
\newtheorem{Proposition}{Proposition}[section]
\newtheorem{Lemma}[Theorem]{Lemma}
\newtheorem{Corollary}[Theorem]{Corollary}
\newtheorem{Fact}[Theorem]{Fact}
\theoremstyle{definition}
\newtheorem{Definition}[Theorem]{Definition}
\theoremstyle{remark}
\newtheorem{Remark}[Theorem]{Remark}

\newcommand{\R}{\mathbb{R}}
\newcommand{\Z}{\mathbb{Z}}
\newcommand{\C}{\mathbb{C}}
\newcommand{\eps}{\varepsilon}
\newcommand{\op}{\operatorname}

\title{A simple construction of smooth counterexamples to the Loewner and Carath\'eodory
conjectures}

\author{Naoya Ando}
\address[Naoya Ando]{Department of Mathematics, Faculty of Advanced Science and Technology,
Kumamoto University, 2-39-1 Kurokami, Chuo-ku, Kumamoto 860-8555, Japan}
\email{andonaoya@kumamoto-u.ac.jp}

\author{Masaaki Umehara}
\address[Masaaki Umehara]{Department of Mathematical and Computing Sciences,
Institute of Science Tokyo, 2-12-1-W8-34, O-okayama, Meguro-ku,
Tokyo 152-8552, Japan}
\email{umehara@comp.isct.ac.jp}


\usepackage[active]{srcltx}

\begin{document}

\begin{abstract}
In a recent announcement, Alp\"oge presented smooth counterexamples
to the Loewner and Carath\'eodory conjectures and credited John-Paul Smith. We present a simpler construction
of a seed by using trigonometric functions. Using this, for every
positive integer~$k$, we construct a smooth real-valued function on
the plane, flat at the origin, whose Hessian eigenline fields share the origin as an
isolated singularity of index~$1+k/2$. A support-function parametrization
realizes this index as an umbilic index of a non-real-analytic
$C^\infty$-regular surface
germ in Euclidean $3$-space. The same framework, applied to a
rapidly decreasing function on the whole plane, gives a strictly
convex closed regular  surface with exactly one umbilic, of index~$2$.
This paper is self-contained with regard to the existence of
an isolated umbilic of index $2$
and the construction of a counterexample to the Carath\'eodory conjecture.
\end{abstract}

\maketitle

\section*{Introduction}
An umbilic is a point of a regular surface in Euclidean $3$-space
$\R^3$ at which the two principal
curvatures coincide. Away from umbilics, the principal directions define
two line fields orthogonal to each other. The index of an isolated umbilic expresses
the change of a continuous angle of either field along a small positively
oriented coordinate circle, divided by $2\pi$. It may be a half-integer,
since the angle of a line is defined modulo $\pi$.
The Loewner conjecture asserts that this index is at most~$1$.
The Carath\'eodory conjecture asserts that a closed (strictly) convex regular  surface in
$\R^3$ has at least two umbilics. On a closed strictly
convex regular surface whose umbilics are isolated, the sum of their indices is~$2$. Thus,
if the Loewner conjecture is true, then the
Carath\'eodory conjecture follows.

Alp\"oge recently announced smooth counterexamples to the Loewner and
Carath\'eodory conjectures in an online post \cite{Alpoge2026}.
In the announcement, he credits John-Paul Smith and notes the use of AI.
The local construction gives Hessian eigenline fields with an isolated
singularity of index~$1+k/2$ for every positive integer~$k$.
The announcement also gives a strictly convex closed regular surface
with exactly one umbilic, of index~$2$.

The purpose of the present paper is to give simpler variants of
the constructions in the announcement.
We use trigonometric functions to construct another seed
whose trace-free Hessian is nowhere zero.
For comparison, we recall the seed and the functions used in
the announcement~\cite{Alpoge2026}.
We identify the real plane $\R^2$ with the complex plane $\C$ by
$z=x+iy$ ($x,y\in\R$).
The seed is the real-valued function
\[
 f_{\mathrm{AS}}(z)
 :=-\frac14\cos(2x)+\frac3{10}\cos(2y)
  -\frac1{32}\cos(4y)+\sin x\sin y.
\]
For each positive integer~$k$, the local construction uses a
real-valued function $g^k_{\mathrm{AS}}$ to obtain a regular surface with
an isolated umbilic of index~$1+k/2$ corresponding to $z=0$.
For $z\ne0$, this function is given by
\[
 \begin{gathered}
 g^k_{\mathrm{AS}}(z)
 :=10^{10}+A(|z|)
 f_{\mathrm{AS}}\!\left(\left(\frac{100}{\bar z}\right)^{k/2}\right),
 \\[1mm]
 A(r):=\frac{r^2}{1+r^2}
       \exp\!\left(-r^{-1/4}e^{-r^2}\right),
 \end{gathered}
\]
where $\bar z$ denotes the complex conjugate of~$z$
and $r:=|z|=\sqrt{z\bar z}$.
The function $g^k_{\mathrm{AS}}$
extends smoothly to the origin with
$g^k_{\mathrm{AS}}(0)=10^{10}$.
For $k=2$, the announcement states that $g^2_{\mathrm{AS}}$ extends smoothly
to $\hat \C:=\C\cup \{\infty\}$.
Under the identification of $\hat \C$ with $S^2$,
the function
$g^2_{\mathrm{AS}}$ is the support function of a convex body whose boundary
has exactly one umbilic.

On the other hand, for $t\in\R$, we define a function
$a(t)$ and our seed $F$ by
\begin{equation}\label{eq:F}
 F(z):=a(x)-a(y)+\sin x\sin y,
 \qquad a(t):=\sin^4t+3\sin^2t.
\end{equation}
One may alternatively take $a(t)=\sin^4t$ in order to obtain the following simpler seed\begin{equation}\label{eq:F1}
 \widetilde F(z):=\sin^4x-\sin^4y+\sin x\sin y.
\end{equation}
All the constructions below can be applied to this alternative seed
(see Remark~\ref{Rmk:A}).
We use $a$ in \eqref{eq:F} because it gives the simpler
formula $a''(t)=6-16\sin^4t$, which makes it easier to verify that
the trace-free Hessian of $F$ is nowhere zero.

In place of the functions $g^k_{\mathrm{AS}}$
($k=1,2,\ldots$), we consider the family of functions
\begin{equation}\label{eq:H}
 H_k(z):=
 \begin{cases}
  \exp\bigl(-|z|^{-k/4}\bigr)
       F\bigl(z^{-k/2}\bigr),&z\ne0,\\[1mm]
  0,&z=0.
 \end{cases}
\end{equation}
Each $H_k$ is smooth and flat at the origin, which means that it
and all its partial derivatives vanish there.
Both $g^k_{\mathrm{AS}}$ and $H_k$ are single-valued even when $k$ is odd.

Let $\mathcal E_{H_k}$ denote either eigenline field of the
Hessian of $H_k$.
In Section~\ref{sec:local-index}, we prove that the origin $(0,0)$ is an
isolated singularity of the field
and that the index at the origin
is
\begin{equation}\label{eq:159}
 \operatorname{ind}_{(0,0)}(\mathcal E_{H_k})=1+\frac{k}{2}.
\end{equation}

We can give a Lorentz-Minkowski
analogue of the above fact using the space-like Ribaucour correspondence
of~\cite{AU2023}, which gives a germ of space-like regular surface with
an isolated umbilic of the same index.

For the global construction in Section~\ref{sec:one-umbilic}, we use
the same seed $F$ and define
\[
 f_\delta(W):=e^{-\delta(1+|W|^2)^{1/4}}F(W)
\]
for a complex variable $W$.
For sufficiently small $\delta>0$, the trace-free Hessian of
$f_\delta$ is nowhere zero. Since $f_\delta$ and all its derivatives
decrease faster than every negative power of $|W|$ as $|W|\to\infty$, the
function
\[
 \rho_\eps(W):=1+\eps\frac{f_\delta(W)}{1+|W|^2}
\]
extends smoothly across the point at infinity $W=\infty$
and for sufficiently small $\eps>0$, it
induces a strictly convex closed non-real-analytic
$C^\infty$-regular surface with exactly one
umbilic of index $2$ as shown in Section 2.
Our approach is related to the authors' earlier construction with
Fujiyama~\cite{AFU} of examples that are $C^1$-regular at the origin and have
arbitrarily large indices.
In our constructions of counterexamples to the Loewner and
Carath\'eodory conjectures, we use support functions and describe
the correspondence between Hessian eigenline fields and principal
line fields directly in Fact~\ref{lem:support-hessian-correspondence}.

It should be remarked that,
in the above-mentioned construction in \cite{AFU}, the authors used Ribaucour parameters rather than
the support-function method employed in the present paper.
Both methods, however, convert principal curvature line fields into the Hessian eigenline fields of an associated function. We note that the
Hessian eigenline fields arising from these two methods are essentially the same.
In fact,
Proposition~\ref{prop:support-ribaucour} in
Appendix~\ref{app:support-ribaucour} identifies our support-function
parametrization with the classical Ribaucour parametrization
through a local change of coordinates,
and this change is characterized by a local
Legendre transform.

We used ChatGPT (OpenAI, GPT-6.1) to assist in the search for candidate
functions for our constructions
and in revising the exposition. The authors have verified all
arguments presented in this paper and take full responsibility for its
mathematical contents.

\section{Local index computation}\label{sec:local-index}

For a real function $h=h(x,y)$, we use the notation
\[
 \nabla h:=\begin{pmatrix}h_x\\h_y\end{pmatrix},\qquad
 \op{Hess}_{x,y}h:=\begin{pmatrix}h_{xx}&h_{xy}\\h_{xy}&h_{yy}\end{pmatrix}.
\]
These are the Euclidean gradient and Hessian, respectively.
The trace
$$
\operatorname{tr}\op{Hess}_{x,y} h=h_{xx}+h_{yy}
$$ is the
sum of the diagonal entries. The {\it trace-free Hessian} is
$$
\op{Hess}_{x,y}h-\tfrac12(\operatorname{tr}\op{Hess}_{x,y}h)I_2,
$$
which has the same eigenline fields as $\op{Hess}_{x,y}h$ wherever this matrix
is not scalar.

We set $W:=X+iY$ ($X,Y\in \R$)
and
\begin{equation}\label{eq:EW}
 \Psi(W):=e^{-\sqrt{|W|}}F(W).
\end{equation}
For each positive integer $k$, set
\[
 (W=)\,W(z)=z^{-k/2},\qquad z\ne0,
\]
choosing a local branch when $k$ is odd.
By definition,
$|W|\to\infty$ as $z\to0$.
Since
\[
 |W(z)|=|z|^{-k/2},
 \qquad \sqrt{|W(z)|}=|z|^{-k/4},
\]
we have
\[
 H_k(z)
 =e^{-|z|^{-k/4}}F(z^{-k/2})
 =\Psi(W(z)).
\]
Since $F(-W)=F(W)$, we also have $\Psi(-W)=\Psi(W)$,
so $H_k$ is a single-valued function.

\subsection*{The index representation}
Throughout this paper, we use the notation
$$
\partial_z:=\tfrac12(\partial_x-i\partial_y),\qquad
\partial_{\overline{z}}:=\tfrac12(\partial_x+i\partial_y),
$$
where $\partial_x$ and
$\partial_y$ denote the real partial derivative operators
with respect to the variables $x,y$.
For a real $C^2$ function $h$, consider the complex-valued
function given by
\begin{equation}\label{eq:local-q}
 q_h:=h_{xx}-h_{yy}+2ih_{xy}=4\overline{h_{zz}},
\end{equation}
where $h_{zz}:=\partial_z^2h$.
Then the Hessian of $h$ is a scalar multiple of $I_2$
exactly where $q_h$ vanishes.
At a point where $q_h\ne0$, write $q_h=|q_h|e^{i\varphi}$, where
$\varphi$ is a local continuous argument. The two Hessian eigenlines
make angles $\varphi/2$ and $\varphi/2+\pi/2$ with respect to the positive $x$-axis.
We denote either Hessian eigenline field by $\mathcal E_h$.
If the origin is an isolated zero of $q_h$, then
(cf. \cite[Appendix~B]{AFU})
\begin{equation}\label{eq:local-cartesian-index}
 \operatorname{ind}_{(0,0)}(\mathcal E_h)
 =\tfrac12\operatorname{wind}(q_h)
 =-\tfrac12\operatorname{wind}(h_{zz}).
\end{equation}
Here $\operatorname{wind}(q_h)$ denotes the winding number
about zero of the loop $\theta\mapsto q_h(re^{i\theta})$ $(0\le\theta\le2\pi)$,
where $r>0$ is sufficiently small.
Let $F(x,y)$ be as in  \eqref{eq:F}.

\begin{Lemma}\label{lem:local-seed}
The function $F$ given in \eqref{eq:F}
and each of its partial derivatives are bounded
on $\C$.
Moreover, there exists a constant $c>0$ such that
\begin{equation}\label{eq:11}
|F_{zz}(z)|\ge c \qquad (z\in\C).
\end{equation}
In particular, $F_{zz}$ does not vanish anywhere on $\C$.
\end{Lemma}

\begin{proof}
The function $F$ is $C^\infty$-differentiable and $2\pi$-periodic in each variable, so it and
all its partial derivatives are bounded.
By \eqref{eq:F}, we have
$
 a''(t)=6-16\sin^4t.
$
Substitution into \eqref{eq:local-q} with $h:=F$ gives
\begin{equation}\label{eq:local-seed-q}
 q_F=12-16(\sin^4x+\sin^4y)+2i\cos x\cos y.
\end{equation}
If $\operatorname{Im}q_F=0$, then $\cos x\cos y=0$, hence
\begin{equation}\label{eq:313}
 \sin^4x+\sin^4y\ge1,
 \qquad \operatorname{Re}q_F\le-4<0.
\end{equation}
Consequently $q_F$ never vanishes. Periodicity gives
\begin{equation}\label{eq:local-seed-complex-bounds}
\inf_{z\in\C}|F_{zz}(z)|>0.
\end{equation}
\end{proof}

\begin{Remark}\label{Rmk:A}
If we set  $a(t)=\sin^4t$ in \eqref{eq:F}, then it gives the seed
$\widetilde F$ given in \eqref{eq:F1}, which also
satisfies Lemma~\ref{lem:local-seed} as well as $F$.
Indeed, $\operatorname{Im}q_{\widetilde F}=2\cos x\cos y$.
If this vanishes, we may assume by symmetry that $\cos x=0$.
Setting $s=\sin^2y$, we obtain
\[
 \operatorname{Re}q_{\widetilde F}
 =-4+12s-16s^2
 =-16\left(s-\frac38\right)^2-\frac74<0.
\]
Thus $\widetilde F_{zz}$ is nowhere zero, and periodicity gives
$\inf_{z\in \mathbb C}|\widetilde F_{zz}(z)|>0$.
Both the local and global constructions below can be applied
to this seed as well.
In the present paper, we use $a(t)=\sin^4t+3\sin^2t$ because
the estimates in \eqref{eq:local-seed-complex-bounds}
are easier to prove.
\end{Remark}

\subsection*{Derivative estimates for \texorpdfstring{$\Psi$}{Psi}}

Let $F(W)$ be as in \eqref{eq:F} with $W:=X+iY$.
For $W\ne0$, set
\begin{equation}\label{eq:psi-radius-damping}
 R:=|W|,\qquad b(W):=e^{-\sqrt R}.
\end{equation}
Then $\Psi$ given in
\eqref{eq:EW} is expressed as
\begin{equation}\label{eq:psi-definition}
 \Psi(W)=b(W)F(W).
\end{equation}
We set
\[
 \partial_W:=\tfrac12(\partial_X-i\partial_Y),\qquad
 \partial_{\overline W}:=\tfrac12(\partial_X+i\partial_Y).
\]

We have
\begin{equation}\label{eq:psi-radius-derivatives}
 R_W(=\partial_WR)=\frac{R}{2W},
 \qquad
 (\sqrt R)_W(=\partial_W\sqrt{R})=\frac{\sqrt R}{4W}.
\end{equation}
Therefore we have
\begin{equation}\label{eq:simple-damping}
 \frac{b_W}{b}=-\frac{\sqrt R}{4W},
 \qquad
 \frac{b_{WW}}{b}=\frac{R+3\sqrt R}{16W^2}.
\end{equation}
For a positive-valued function $v$, the notation $u=O(v)$
means that $u/v$ is bounded in the stated limit, and $u=o(v)$
means that $u/v$ tends to zero.
Since $|W|=R$,
\eqref{eq:simple-damping} gives
\begin{equation}\label{eq:psi-damping-asymptotics}
 \frac{b_W}{b}=O(R^{-1/2}),
 \qquad
 \frac{b_{WW}}{b}=O(R^{-1})
 \qquad(R\to\infty).
\end{equation}

By \eqref{eq:psi-definition},
we have
\begin{equation}\label{eq:psi-product-rule}
 \frac{\Psi_W}{b}
 =F_W+\frac{b_W}{b}F,\qquad
 \frac{\Psi_{WW}}{b}
 =F_{WW}+2\frac{b_W}{b}F_W+\frac{b_{WW}}{b}F.
\end{equation}
Since $F$ and $F_W$ are bounded,
\eqref{eq:psi-damping-asymptotics} yields that
\begin{equation}\label{eq:simple-seed-asymptotics}
 \frac{\Psi_W}{b}=O(1),
 \qquad
 \frac{\Psi_{WW}}{b}
 =F_{WW}+O(R^{-1/2})
 \qquad(R\to\infty).
\end{equation}
Thus,
there exist constants
\begin{equation}\label{eq:C3}
C_1,C_2>0\qquad R_0\ge1
\end{equation}
satisfying
\begin{align}\label{eq:psi-second-derivative-error}
 &\left|
 \frac{\Psi_{WW}(W)}{b(W)}-F_{WW}(W)
 \right|
 \le C_1R^{-1/2},
\\
\label{eq:psi-first-derivative-bound}
 &\left|\frac{\Psi_W(W)}{b(W)}\right|\le C_2
\qquad \qquad(R=|W|\ge R_0).
\end{align}
By \eqref{eq:11} with \eqref{eq:psi-second-derivative-error},
we observe that, for
$R\ge\max\{R_0,(2C_1/c)^2\}$,
\begin{equation}\label{eq:psi-second-derivative-lower-bound}
 \begin{aligned}
 \left|\frac{\Psi_{WW}(W)}{b(W)}\right|
 &\ge |F_{WW}(W)|
     -\left|\frac{\Psi_{WW}(W)}{b(W)}-F_{WW}(W)\right|\\
 &\ge c-C_1R^{-1/2}
 \ge\frac{c}{2}>0,
 \end{aligned}
\end{equation}
which implies that $\Psi_{WW}(W)\ne0$
for all sufficiently large $|W|$.

\subsection*{Smoothness and the index}
For every positive integer $k$, consider the function $H_k$
defined on $\C(=\R^2)$ by
\begin{equation}\label{eq:local-family}
 H_k(z):=
 \begin{cases}
 e^{-|z|^{-k/4}}F(z^{-k/2}),&z\ne0,\\
 0,&z=0.
 \end{cases}
\end{equation}

\begin{Theorem}\label{thm:local-all-indices}
The function $H_k$ is single-valued, smooth, and flat at the origin $(0,0)$.
Its Hessian has distinct eigenvalues on a sufficiently small punctured
disk, and
\begin{equation}\label{eq:local-all-index-result}
 \operatorname{ind}_{(0,0)}(\mathcal E_{H_k})=1+\frac{k}{2}.
\end{equation}
\end{Theorem}

\begin{proof}

Fix a positive integer $k$.
Since $F$ is even, $H_k$ is independent of the branch of $z^{-k/2}$.
We first verify smoothness and flatness
of the function $H_k$. All derivatives of $F$ are bounded.
Since the exponential factor makes the
right-hand side of \eqref{eq:local-family} tend to zero faster than every power of $|z|$, we can
conclude that $H_k$ is smooth and flat at $(0,0)$.

\medskip
By setting
\begin{equation}\label{eq:W500}
 (W=)\,W(z)=z^{-k/2},
\end{equation}
we have
$$
H_k(z)=(\Psi\circ W)(z),
$$
where $\circ$ denotes the composition of maps.
Since $W(z)$ is holomorphic, its complex conjugate
satisfies
\begin{equation}\label{eq:W0}
(\overline W)_{z}=0.
\end{equation}
Although $\Psi$ depends on both $W$ and $\overline W$,
we have
\[
 (H_k)_z
 =\Psi_W W_z+\Psi_{\overline W}\,(\overline W)_z
 =\Psi_W W_z.
\]
Applying \eqref{eq:W0} again when differentiating $\Psi_W$, we obtain
\[
 (H_k)_{zz}
 =\Psi_{WW}(W_z)^2+\Psi_W W_{zz}.
\]
Direct differentiations give
\[
 W_z=-\frac{k}{2}z^{-k/2-1},
 \qquad
 W_{zz}=\frac{k}{2}\left(\frac{k}{2}+1\right)z^{-k/2-2},
\]
and hence
\[
 (W_z)^2=\frac{k^2}{4}z^{-k-2},
 \qquad
 \frac{W_{zz}}{(W_z)^2}
 =\left(1+\frac2k\right)\frac1W.
\]
Thus
\begin{equation}\label{eq:simple-chain-rule}
 (H_k)_{zz}
 =\frac{k^2}{4}z^{-k-2}
 \left(
 \Psi_{WW}(W)+\left(1+\frac2k\right)\frac{\Psi_W(W)}{W}
 \right).
\end{equation}
Set
\begin{equation}\label{eq:local-eta-definition}
 \eta_k(W):=
 \left(
 \frac{\Psi_{WW}(W)}{b(W)}-F_{WW}(W)
 \right)
 +\left(1+\frac2k\right)\frac{\Psi_W(W)}{b(W)W}.
\end{equation}
We use the same constants $C_1,C_2$ and $R_0$
as in \eqref{eq:C3} which satisfy
\eqref{eq:psi-second-derivative-error} and
\eqref{eq:psi-first-derivative-bound}.
Since $1+2/k\le3$ and $R^{-1}\le R^{-1/2}$ for $R\ge1$,
\[
 \begin{aligned}
 |\eta_k(W)|
 &\le
 \left|\frac{\Psi_{WW}(W)}{b(W)}-F_{WW}(W)\right|
 +\left(1+\frac2k\right)\frac1{|W|}
       \left|\frac{\Psi_W(W)}{b(W)}\right|\\
 &\le C_1R^{-1/2}+3C_2R^{-1}
\le C_\eta R^{-1/2},
 \end{aligned}
\]
where
$$
R\ge R_0,\qquad  C_\eta:=C_1+3C_2.
$$
This constant $C_\eta$ is independent of the argument of $W$ and of $k$.
Therefore, we have
\begin{equation}\label{eq:local-eta-bound}
 |\eta_k(W)|
\le C_\eta R^{-1/2}
 \qquad(R\ge R_0).
\end{equation}

Factoring $b(W)$ out of \eqref{eq:simple-chain-rule}, we obtain
\begin{equation}\label{eq:simple-normalized-hessian}
 (H_k)_{zz}
=\frac{k^2}{4}b(W)z^{-k-2}
       \bigl(F_{WW}(W)+\eta_k(W)\bigr).
\end{equation}

We set
\begin{equation}\label{eq:local-general-leading}
 L_k(z):=\frac{k^2}{4}b(W)z^{-k-2}F_{WW}(W).
\end{equation}
Equations \eqref{eq:simple-normalized-hessian} and
\eqref{eq:local-general-leading} give the decomposition
\begin{equation}\label{eq:local-leading-error-decomposition}
 (H_k)_{zz}
 =L_k(z)+\frac{k^2}{4}b(W)z^{-k-2}\eta_k(W).
\end{equation}
We remark that $L_k$ will play the role of the leading term.
We now estimate the second term of the right-hand side of
\eqref{eq:local-leading-error-decomposition}
as a relative error term.
Noticing $L_k(z)\ne0$ and using
\eqref{eq:11} and \eqref{eq:local-eta-bound},
we obtain
\begin{equation}\label{eq:local-general-relative-error}
 \frac{|(H_k)_{zz}-L_k|}{|L_k|}
=\frac{|\eta_k(W)|}{|F_{WW}(W)|}
\le\frac{C_\eta}{c}R^{-1/2}
 =O(|z|^{k/4}).
\end{equation}
Choose a radius $\rho_k\in(0,1)$ satisfying
\[
 |z|^{-k/2}\ge R_0,
 \qquad
 \frac{C_\eta}{c}|z|^{k/4}<\frac12
 \qquad(0<|z|<\rho_k).
\]
Then \eqref{eq:local-general-relative-error} implies
\[
 |(H_k)_{zz}(z)-L_k(z)|<\frac12|L_k(z)|
 \qquad(0<|z|<\rho_k).
\]
For $0<|z|<\rho_k$, define a continuous deformation from $L_k(z)$
to $(H_k)_{zz}(z)$ by the line segment
\begin{equation}\label{eq:local-comparison-homotopy}
 G_{k,t}(z)
 :=L_k(z)+t\Bigl((H_k)_{zz}(z)-L_k(z)\Bigr)
 \qquad(0\le t\le1).
\end{equation}
Then, we have $G_{k,0}=L_k$ and $G_{k,1}=(H_k)_{zz}$.
Both $L_k$ and $(H_k)_{zz}$ are single-valued and continuous.
Hence $G_{k,t}(z)$ is continuous in $(z,t)$.
This segment does not pass through zero:
\begin{align}\label{eq:local-comparison-nonvanishing}
|G_{k,t}(z)|
&\ge |L_k(z)|-t|(H_k)_{zz}(z)-L_k(z)|\\
&>\frac12|L_k(z)|>0 \qquad (0<|z|<\rho_k,\,\,0\le t\le 1).
\nonumber
\end{align}
We use \eqref{eq:local-comparison-nonvanishing} in two ways.
First, at $t=1$ we have $G_{k,1}=(H_k)_{zz}$, so
$(H_k)_{zz}\ne0$ throughout $0<|z|<\rho_k$.
By \eqref{eq:local-q}, $(H_k)_{zz}\ne0$ is
equivalent to the condition that the Hessian matrix of $H_k(x,y)$ has two
distinct eigenvalues.
Since $\op{Hess}_{x,y}H_k(0)=0$, the origin is therefore an isolated
singularity of the Hessian eigenline fields.
Second, fix $0<r<\rho_k$ and compute the winding numbers on the
counterclockwise circle $z=re^{i\theta}$, with
$0\le\theta\le2\pi$, using the representation stated in
\eqref{eq:local-cartesian-index}.
By \eqref{eq:local-comparison-homotopy} and
\eqref{eq:local-comparison-nonvanishing},
we can define the following
continuous family of curves
\[
\R/2\pi\Z\ni \theta\mapsto G_{k,t}(re^{i\theta})\in \C\setminus \{0\}
\qquad (0\le t\le1).
\]
At $t=0$ the curve is given by $L_k$, and at $t=1$ it is given
by $(H_k)_{zz}$.
Their winding numbers agree because the winding number is unchanged
under a continuous deformation through curves that avoid zero.
Consequently,
\begin{equation}\label{eq:local-winding-comparison}
 \operatorname{wind}((H_k)_{zz})
 =\operatorname{wind}(L_k).
\end{equation}
By the definition of $L_k$ (cf. \eqref{eq:local-general-leading}), the
positive real factor $b(W)$ contributes no winding, and hence
\begin{equation}\label{eq:local-general-leading2}
\operatorname{wind}(L_k)
=-k-2+\operatorname{wind}\bigl((F_{WW}\circ W)(z)\bigr).
\end{equation}
Since $F_{WW}$ is even and nowhere zero, we can conclude that
$\operatorname{wind}(L_k)=-k-2$.
Finally, \eqref{eq:local-cartesian-index} gives
\[
 \operatorname{ind}_{(0,0)}(\mathcal E_{H_k})
 =-\frac12\operatorname{wind}((H_k)_{zz})
 =\frac{k+2}{2}
 =1+\frac{k}{2}.
\]
\end{proof}

\subsection*{Realization by support functions}
We use support functions in order to realize Hessian eigenline fields as
principal line fields of regular surfaces in $\R^3$.
For a smooth real-valued function $\rho$ on an open subset of
the unit sphere $S^2\subset\R^3$, set
\begin{equation}\label{eq:support-parametrization}
 \bm{X}_{\rho}(\bm{n}):=\rho(\bm{n})\bm{n}+\operatorname{grad}_{S^2}\rho(\bm{n})\quad (\bm{n}\in S^2),
\end{equation}
where $\operatorname{grad}_{S^2}\rho$ is the gradient of $\rho$
with respect to the Riemannian metric $g$ of the unit sphere $S^2$.
At each point $\bm{n}$ in the domain of $\rho$,
$\operatorname{grad}_{S^2}\rho(\bm{n})$ is the unique tangent vector satisfying
\[
 g_{\bm{n}}\bigl((\operatorname{grad}_{S^2}\rho)(\bm{n}),\bm{v}\bigr)
 =\bm{v}(\rho)
 \qquad (\bm{v}\in T_{\bm{n}}S^2).
\]
Here $\bm{v}(\rho)$ is the directional derivative of $\rho$
in the direction $\bm{v}$.

Define a self-adjoint endomorphism
$\mathcal R_{\rho}$ on $S^2$ by
\begin{equation}\label{eq:support-radius-operator}
 g_{\bm{n}}(\mathcal R_{\rho} \bm{v},\bm{w})
 :=(\operatorname{Hess}_{S^2}\rho)(\bm{v},\bm{w})+\rho\,g_{\bm{n}}(\bm{v},\bm{w})
 \qquad (\bm{v},\bm{w}\in T_{\bm{n}}S^2),
\end{equation}
where $\operatorname{Hess}_{S^2}\rho$ is the Hessian of $\rho$
with respect to the Levi-Civita connection of $S^2$.
When the operator $\mathcal R_\rho$
is invertible, $\bm{X}_\rho$ is regular with unit normal $\bm{n}$.
Let
$S$ denote the shape operator of $\bm{X}_\rho$
 with respect to this normal,
defined as the negative of the differential of the Gauss map.
Suppose that $\mathcal R_\rho$ is invertible.
The pullback of $S$ by $\bm{X}_\rho$ is $-\mathcal R_\rho^{-1}$.
Indeed, differentiating \eqref{eq:support-parametrization},
the normal components cancel because
$g_{\bm{n}}(\operatorname{grad}_{S^2}\rho,\bm{v})=\bm{v}(\rho)$. Thus
\[
 d\bm{X}_\rho(\bm{v})
 =\rho\bm{v}
  +\nabla^{S^2}_{\bm{v}}\operatorname{grad}_{S^2}\rho
 =\mathcal R_\rho\bm{v},
\]
where $\nabla^{S^2}$ denotes the Levi-Civita connection of $S^2$.
The fact that the
Gauss map composed with $\bm{X}_\rho$ is the identity gives
$S\,d\bm{X}_\rho=-I$. Consequently,
the pullback of $S$ by $\bm{X}_\rho$ is
just $-\mathcal R_\rho^{-1}$.
The principal line fields of $\bm{X}_\rho$ therefore
are pulled back to the eigenline fields
of $\mathcal R_\rho$.

\noindent For $W=X+iY$, the map
\[
 \bm{\nu}(W):=\frac{(2X,2Y,|W|^2-1)}{1+|W|^2}
\]
identifies $\C$ with $S^2\setminus\{\bm{N}\}$, where
$\bm{N}:=(0,0,1)$ corresponds to $W=\infty$.
We set
\begin{equation}\label{eq:support-round-metric}
 D:=1+|W|^2,\quad
 g=\lambda^2(dX^2+dY^2),\quad \lambda:=\frac2D.
\end{equation}

\begin{Lemma}\label{lem:support-umbilic-criterion}
Let $\rho$ be a smooth real-valued function on an open subset of
$S^2$. A regular point of
$\bm{X}_\rho$ is an umbilic if and only if
at this point, $\rho$ satisfies
\begin{equation}\label{eq:support-umbilic-condition}
 \rho_{WW}+\frac{2\overline W}{D}\rho_W=0.
\end{equation}
\end{Lemma}

\begin{proof}
Set
\begin{equation}\label{eq:support-radius-matrix}
 (\mathcal R_\rho\partial_X,\mathcal R_\rho\partial_Y)
 =(\partial_X,\partial_Y)\begin{pmatrix}a&b\\ b&c\end{pmatrix},
 \qquad T:=\operatorname{Hess}_{S^2}\rho+\rho g.
\end{equation}
By \eqref{eq:support-radius-operator},
\[
 \begin{aligned}
 T(\partial_X,\partial_X)
 &=g(\mathcal R_\rho\partial_X,\partial_X)=\lambda^2a,\quad
 T(\partial_X,\partial_Y)
=g(\mathcal R_\rho\partial_X,\partial_Y)=\lambda^2b,\\
 T(\partial_Y,\partial_Y)
 &=g(\mathcal R_\rho\partial_Y,\partial_Y)=\lambda^2c.
 \end{aligned}
\]
Then
\begin{equation}\label{eq:support-tensor-complex-components}
 \begin{split}
 T(\partial_W,\partial_W)
 &=\frac14\Bigl(T(\partial_X,\partial_X)
       -T(\partial_Y,\partial_Y)-2iT(\partial_X,\partial_Y)\Bigr) \\
&=\frac{a-c-2ib}{D^2}.
 \end{split}
\end{equation}
On the other hand,
since
$g=\lambda^2(dX^2+dY^2)$, we have
$g(\partial_W,\partial_W)=0$. Thus, we have
\begin{equation}\label{eq:support-tensor-complex-hessian}
 T(\partial_W,\partial_W)
 =(\operatorname{Hess}_{S^2}\rho)(\partial_W,\partial_W)
 =\rho_{WW}+\frac{2\overline W}{D}\rho_W.
\end{equation}
Since $\lambda>0$ and $a,b,c$ are real,
by \eqref{eq:support-tensor-complex-components} and
\eqref{eq:support-tensor-complex-hessian}, \eqref{eq:support-umbilic-condition} holds
if and only if $a=c$ and $b=0$, that is, if and only if
$\mathcal R_\rho$ is scalar. At a regular point of $\bm{X}_\rho$,
$\mathcal R_\rho$ is invertible, and the pullback of $S$
to $S^2$ is $-\mathcal R_\rho^{-1}$. Thus $\mathcal R_\rho$ is scalar if and
only if $S$ is scalar, which is precisely the
umbilic condition.
\end{proof}

The following fact is based on a classical formula of Bonnet
for principal curvature lines
(cf. Guti\'errez and S\'anchez-Bringas
\cite[Proposition~2.1]{GS1997}).
For the reader's convenience, we provide a proof.
Appendix~\ref{app:support-ribaucour} describes
the relation between the support-function parametrization and
the classical Ribaucour representation
(Fact~A.1 in \cite[Appendix~A]{AFU}),
and gives an alternative explanation of the correspondence between
Hessian eigenline fields and principal line fields.

\begin{Fact}\label{lem:support-hessian-correspondence}
Let $U$ be a smooth real-valued function on an open subset of the
$W$-plane, and let $\eps\ne0$. For
\begin{equation}\label{eq:support-transform}
 \rho:=1+\frac{\eps U}{D},
\end{equation}
we have
\begin{equation}\label{eq:support-complex-cancellation}
 \rho_{WW}+\frac{2\overline W}{D}\rho_W
 =\frac{\eps}{D}U_{WW}.
\end{equation}
Consequently, by Lemma~\ref{lem:support-umbilic-criterion},
a regular point of $\bm{X}_\rho$ is an umbilic if and only if
$U_{WW}=0$.
At regular nonumbilics, the principal line fields of
$\bm{X}_\rho$, expressed in the $(X,Y)$-coordinates,
are the Hessian eigenline fields of $U$.
\end{Fact}

\begin{proof}
Since $D_W=\overline W$ and $\partial_W\overline W=0$, differentiating
\eqref{eq:support-transform} gives
\begin{equation}\label{eq:support-transform-derivatives}
 \rho_W D^2=\eps(DU_W-\overline W\,U),\qquad
 (\rho_W D^2)_W=\eps D U_{WW}.
\end{equation}
Since
\[
 \frac{(\rho_W D^2)_W}{D^2}
 =\rho_{WW}+\frac{2\overline W}{D}\rho_W,
\]
dividing the second identity of \eqref{eq:support-transform-derivatives}
by $D^2$ proves \eqref{eq:support-complex-cancellation}.

Equations \eqref{eq:support-tensor-complex-components},
\eqref{eq:support-tensor-complex-hessian} and
\eqref{eq:support-complex-cancellation} give
\[
 a-c-2ib=\eps D U_{WW}.
\]
Since $4U_{WW}=U_{XX}-U_{YY}-2iU_{XY}$, comparisons of the real and
imaginary parts yield
\[
 a-c=\frac{\eps D}{4}(U_{XX}-U_{YY}),\qquad
 b=\frac{\eps D}{4}U_{XY}.
\]
Therefore, if we denote by $\widetilde{\mathcal R}_\rho$ the representation
matrix of $\mathcal R_\rho$ in \eqref{eq:support-radius-matrix}, then
\begin{equation}\label{eq:support-trace-free-correspondence}
 \widetilde{\mathcal R}_\rho-\frac{\operatorname{tr}\widetilde{\mathcal R}_\rho}{2}I_2
 =\frac{\eps D}{4}
 \left(
 \operatorname{Hess}_{X,Y}U-\frac{U_{XX}+U_{YY}}{2}I_2
 \right).
\end{equation}
Thus $\widetilde{\mathcal R}_\rho$ and
$\operatorname{Hess}_{X,Y}U$ have the same eigenline fields wherever
they are not scalar. At regular points, the pulled-back principal
line fields of $\bm{X}_\rho$ are the eigenline fields of
$\mathcal R_\rho$.
Therefore, by \eqref{eq:support-trace-free-correspondence},
the pulled-back principal line fields are the Hessian eigenline fields of $U$.
\end{proof}

\begin{proof}[Alternative proof under an additional invertibility assumption]
\leavevmode\par\nobreak\noindent
For the principal line fields, we use the classical Ribaucour
correspondence and the local Legendre transform.
Set
\[
 \psi:=\frac{D\rho}{2}=\frac D2+\frac{\eps U}{2},
 \qquad
 \operatorname{Hess}_{X,Y}\psi
   =I_2+\frac{\eps}{2}\operatorname{Hess}_{X,Y}U.
\]
Near each point where $\operatorname{Hess}_{X,Y}\psi$ is invertible,
Proposition~\ref{prop:support-ribaucour} gives the local identity
\[
 \bm{P}_h\bigl(\nabla\psi(\bm{p})\bigr)
 =\bm{X}_\rho(\bm{p}),
\]
where $h$ is the local Legendre transform of $\psi$.
By Proposition~\ref{prop:ribaucour-eigenlines},
the principal line fields of $\bm{X}_\rho$
at regular nonumbilics, expressed in the $(X,Y)$-coordinates,
are the Hessian eigenline fields of $\psi$.
Since $\eps\ne0$, these are the Hessian eigenline fields of $U$.
\end{proof}

\begin{Remark}
\label{rmk:main-legendre-interpretation}
As mentioned in the introduction,
in the previous work  \cite{AFU},
the authors used Ribaucour parameters rather than
the support-function method employed in the present paper.
Both methods convert principal curvature line fields into the Hessian eigenline fields of an associated function.
In the appendix, we explain
the relationship between Ribaucour parametrization
and our support-function parametrization
when the Hessian of $\psi=\rho D/2$ is invertible.
\end{Remark}

\begin{Corollary}
\label{cor:local-support-indices}
For every positive integer~$k$,
there exists a $C^\infty$-regular surface germ in $\R^3$
with an isolated umbilic of index $1+k/2$.
\end{Corollary}

\begin{proof}
Fix $\eps>0$ and consider the support function
\[
 \rho(z)=1+\eps\frac{H_k(z)}{1+|z|^2}.
\]
Since $H_k$ is flat at the origin, the operator
$\mathcal R_\rho$ is the identity there. By continuity, it is invertible on a
sufficiently small neighborhood of the origin.
Thus $\bm{X}_\rho$ is regular there.
Theorem~\ref{thm:local-all-indices} and
Fact~\ref{lem:support-hessian-correspondence}
give the required isolated umbilic and its index.
\end{proof}

The space-like Ribaucour correspondence of Ando--Umehara
\cite[Theorem~3.3]{AU2023} gives the same local conclusion in
Lorentz-Minkowski $3$-space $\mathbb L^3$, that is, $\R^3$
with the metric $dX^2+dY^2-dZ^2$ in real coordinates $(X,Y,Z)$.
A regular surface is space-like when the induced metric is positive definite.
We therefore obtain the following.

\begin{Corollary}
\label{cor:local-ribaucour-indicesL}
For every positive integer~$k$,
there exists a $C^\infty$-regular space-like surface germ in
$\mathbb L^3$ with an isolated umbilic of index $1+k/2$. In particular, the indices of isolated umbilics
of  surface germs in~$\mathbb L^3$ are unbounded above.
\end{Corollary}

\section{Support functions and the global construction}
\label{sec:one-umbilic}

We apply the support-function correspondence of
Fact~\ref{lem:support-hessian-correspondence} to a rapidly decreasing
function on $\C$ constructed from the same seed $F$.
This gives a strictly convex closed surface with exactly
one umbilic.

\subsection*{A rapidly decreasing function on the plane}
For the seed $F$ in \eqref{eq:F}, we set
\begin{equation}\label{eq:support-seed-bounds}
 M_0:=\sup_{W\in\C}|F(W)|,\qquad
 M_1:=\sup_{W\in\C}|F_W(W)|.
\end{equation}
Both $M_0$ and $M_1$ are finite.
For $\delta>0$, define
\begin{equation}\label{eq:support-global-damping}
 \begin{gathered}
 \Phi(W):=(1+|W|^2)^{1/4},\qquad b_\delta(W):=e^{-\delta\Phi(W)},\\
 f_\delta(W):=b_\delta(W)F(W).
 \end{gathered}
\end{equation}
Choose and fix $\delta>0$ such that
\begin{equation}\label{eq:support-delta-choice}
 \mathcal E_\delta
 :=\frac{\delta M_1}{2}
  +\frac{(\delta^2+3\delta)M_0}{16}<\frac c2,
\end{equation}
where $c$ is the constant given in
Lemma~\ref{lem:local-seed}.

\begin{Lemma}\label{lem:rib-damping}
The function $f_\delta$ is $C^\infty$-differentiable on~$\C$ and satisfies
\begin{equation}\label{eq:rib-damped-hessian}
 |(f_\delta)_{WW}(W)|
 \ge\frac{c}{2}b_\delta(W)>0
 \qquad(W\in\C).
\end{equation}
Moreover, $f_\delta$ and all its derivatives decrease faster
than every negative power of~$|W|$ as $|W|\to\infty$.
\end{Lemma}

\begin{proof}
Differentiating $\Phi$  with respect to $W$ gives
\begin{equation}\label{eq:rib-Phi-derivatives}
 \Phi_W=\frac{\overline W}{4(1+|W|^2)^{3/4}},\qquad
 \Phi_{WW}=-\frac{3\overline W^{\,2}}
                    {16(1+|W|^2)^{7/4}}.
\end{equation}
Since $|W|^2\le1+|W|^2$, these formulas imply the uniform estimates
\begin{equation}\label{eq:rib-Phi-uniform-bounds}
 |\Phi_W|\le\frac14,\qquad |\Phi_{WW}|\le\frac3{16}.
\end{equation}
We also have
\[
 \frac{(b_\delta)_W}{b_\delta}=-\delta\Phi_W,\qquad
 \frac{(b_\delta)_{WW}}{b_\delta}
 =\delta^2\Phi_W^2-\delta\Phi_{WW}.
\]
So we have
\begin{equation}\label{eq:rib-damping-product-rule}
 \frac{(f_\delta)_{WW}}{b_\delta}
 =F_{WW}-2\delta\Phi_WF_W
   +(\delta^2\Phi_W^2-\delta\Phi_{WW})F.
\end{equation}
Using \eqref{eq:support-seed-bounds} and \eqref{eq:rib-Phi-uniform-bounds}, we obtain
\begin{equation}\label{eq:rib-relative-error}
 \left|\frac{(f_\delta)_{WW}}{b_\delta}-F_{WW}\right|
 \le\frac{\delta M_1}{2}
      +\frac{(\delta^2+3\delta)M_0}{16}
 =\mathcal E_\delta<\frac c2.
\end{equation}

\noindent As in \eqref{eq:psi-second-derivative-lower-bound},
by \eqref{eq:rib-relative-error}, we obtain

$$
\left|\frac{(f_\delta)_{WW}}{b_\delta}\right|
\ge |F_{WW}|-
|\frac{(f_\delta)_{WW}}{b_\delta}-F_{WW}|
>\frac{c}2.
$$

Finally,
since $F$ and all its derivatives are bounded, and every derivative
of $b_\delta$ is $b_\delta$ times a bounded function, the rapid decay of
$b_\delta(W)$ implies the same decay for $f_\delta$ and all its derivatives.
\end{proof}

\subsection*{The closed surface \texorpdfstring{with a}{with a} unique umbilic}
Set
\begin{equation}\label{eq:support-global-function}
 \rho_\eps(W):=1+\eps\frac{f_\delta(W)}{1+|W|^2}
 \quad(W\in\C).
\end{equation}

\begin{Theorem}\label{thm:support-one-umbilic}
Fix $\delta>0$ satisfying \eqref{eq:support-delta-choice}.
For all sufficiently small $\eps>0$, $\rho_\eps$
gives a smooth function on $S^2$ by the stereographic projection,
and its support-function parametrization $\bm{X}_{\rho_\eps}$ is a
$C^\infty$-embedding onto a strictly convex closed regular surface.
This surface has exactly one umbilic,  and its index is $2$.
\end{Theorem}

\begin{proof}
First we verify smoothness at infinity. We set
$z:=1/W$. Then
\[
 \rho_\eps(1/z)-1
 =\eps\frac{|z|^2}{1+|z|^2}f_\delta(1/z).
\]
By Lemma~\ref{lem:rib-damping}, every derivative
of $f_\delta(1/z)$ tends to zero faster than any power of $|z|$.
Thus $f_\delta(1/z)$ extends by zero to a smooth function flat at
$z=0$.
Consequently \eqref{eq:support-global-function} implies that
$\rho_\varepsilon$ extends smoothly to $S^2$.

Since $\rho_\eps$ is a smooth perturbation of the constant function $1$
on $S^2$, the operator $\mathcal R_{\rho_\eps}$ tends to the identity uniformly and
$\bm{X}_{\rho_\eps}$ tends to $\bm{X}_1$, with all derivatives, as $\eps\to0$.
In addition, for sufficiently small $\eps$,
the map $\bm{X}_{\rho_\eps}$ remains an embedding
and the operator $\mathcal R_{\rho_\eps}$ is positive definite. Its shape operator is
$-\mathcal R_{\rho_\eps}^{-1}$, so, with the inward unit normal,
both principal curvatures are positive.
This means that $\bm{X}_{\rho_\eps}$ is strictly convex.
At every point $W\in\C$, Fact~\ref{lem:support-hessian-correspondence}
and Lemma~\ref{lem:rib-damping}
give
\[
 (\rho_\eps)_{WW}+\frac{2\overline W}{1+|W|^2}(\rho_\eps)_W
 =\frac{\eps}{1+|W|^2}(f_\delta)_{WW}\ne0.
\]
By Lemma~\ref{lem:support-umbilic-criterion}, there are
therefore no umbilics on $\C$.
The point $\infty$ corresponds to
$\bm{N}(=(0,0,1))$, at which the function
$\rho_\eps$ has the same derivatives of every order as the constant
function $1$. Thus $\bm{N}$ is the unique umbilic.
Since the Euler characteristic $\chi(S^2)$ of $S^2$ is equal to $2$,
the Poincar\'e-Hopf theorem for
principal line fields on the sphere gives its index at the umbilic as
$\chi(S^2)=2$.
\end{proof}

\appendix
\section{Ribaucour parametrization and Legendre duality}
\label{app:support-ribaucour}

In this appendix, we recall the local Legendre transform and apply it
to the change from support-function coordinates to Ribaucour parameters.

\subsection*{Local Legendre duality}

In the Legendre transform, we use the first partial derivatives of
a function as new variables. We then define a new function in these
variables.

\begin{Definition}[Local Legendre transform]
\label{def:legendre-transform}
Let $\Omega$ be an open subset of $\R^2$, and let $\psi\in C^\infty(\Omega)$.
Write $\bm{p}=(p_1,p_2)$ and
$\nabla\psi=(\partial\psi/\partial p_1,\partial\psi/\partial p_2)$.
Let $U\subset\Omega$ be an open subset. Suppose that the gradient map
\begin{equation}\label{eq:legendre-coordinate-change}
 \nabla\psi:U\rightarrow V:=\nabla\psi(U),\qquad
 \bm{p}\mapsto\bm{x}=\nabla\psi(\bm{p})
\end{equation}
is a diffeomorphism. We call the function
\begin{equation}\label{eq:legendre-definition}
 h:V\rightarrow\R,\qquad
 h(\bm{x}):=(\nabla\psi)^{-1}(\bm{x})\cdot\bm{x}
             -\psi\bigl((\nabla\psi)^{-1}(\bm{x})\bigr)
\end{equation}
the \emph{local Legendre transform} of $\psi$ on $U$.
\end{Definition}

Fix $\bm{p}\in U$ and set $\bm{x}=\nabla\psi(\bm{p})$.
For $\bm{q}=(q_1,q_2)\in\R^2$, the point $(\bm{q},t)$ lies
in the tangent plane to the graph of $\psi$ at
$(\bm{p},\psi(\bm{p}))$ if and only if
\begin{equation}\label{eq:TQ}
 t=\psi(\bm{p})+\nabla\psi(\bm{p})\cdot(\bm{q}-\bm{p})
   =\bm{x}\cdot\bm{q}-h(\bm{x}),
\end{equation}
which is a linear equation in the variables $t,q_1$ and $q_2$.
Since $x_k=\psi_{p_k}$ for $k=1,2$ and $\bm{x}=(x_1,x_2)$,
differentiating \eqref{eq:legendre-definition} gives
\begin{equation}\label{eq:legendre-duality-identities}
 \nabla h(\bm{x})=(\nabla\psi)^{-1}(\bm{x}).
\end{equation}
For $\bm{x}\in V$, write
$\bm{p}(\bm{x})=(p_1(\bm{x}),p_2(\bm{x}))
=(\nabla\psi)^{-1}(\bm{x})$.
The $i$-th component of $\nabla\psi(\bm{p}(\bm{x}))=\bm{x}$ is
$\psi_{p_i}(\bm{p}(\bm{x}))=x_i$.
Differentiating with respect to $x_j$, we have
\begin{equation}\label{eq:legendre-duality-identities2}
 \sum_{k=1}^2
 \psi_{p_ip_k}\bigl(\bm{p}(\bm{x})\bigr)
 \frac{\partial p_k}{\partial x_j}(\bm{x})
 =\delta_{ij}.
\end{equation}
By \eqref{eq:legendre-duality-identities}, we have
$p_k(\bm{x})=h_{x_k}(\bm{x})$.
Hence $\partial p_k/\partial x_j=h_{x_kx_j}$, and
\eqref{eq:legendre-duality-identities2} says that the product of
$(\operatorname{Hess}_{p_1,p_2}\psi)(\bm{p}(\bm{x}))$ and
$\operatorname{Hess}_{x_1,x_2}h(\bm{x})$ is the identity matrix.
Therefore
\begin{equation}\label{eq:legendre-duality-identities3}
 \operatorname{Hess}_{x_1,x_2}h(\bm{x})
 =\bigl(
   (\operatorname{Hess}_{p_1,p_2}\psi)
   \bigl(\bm{p}(\bm{x})\bigr)
  \bigr)^{-1}.
\end{equation}
The identities \eqref{eq:legendre-duality-identities} and
\eqref{eq:legendre-duality-identities3} are called
the \emph{local Legendre dualities}.

\subsection*{Application to Ribaucour parameters}

We now consider the relation between the support-function
parametrization in stereographic coordinates
and the Ribaucour
representation in Fact~A.1 in \cite[Appendix~A]{AFU}.

Let $\Omega$ be an open subset of $\R^2$,
and let $\rho\in C^\infty(\Omega)$. Set
\begin{equation}\label{eq:legendre-support-data}
 D:=1+|\bm{p}|^2,\qquad
 \psi:=\frac{\rho D}{2}.
\end{equation}
Using the inverse map of the stereographic projection
\[
 \bm{\nu}(\bm{p})
 :=\frac{(2p_1,2p_2,|\bm{p}|^2-1)}{D},
\]
we regard $\rho$ as a function on $\bm{\nu}(\Omega)\subset S^2$
and define
\begin{equation}\label{eq:appendix-support-parametrization}
 \bm{X}_\rho(\bm{p})
 :=\rho(\bm{p})\bm{\nu}(\bm{p})
   +(\operatorname{grad}_{S^2}\rho)(\bm{\nu}(\bm{p})).
\end{equation}

Let $V$ be an open subset of $\Omega$. Suppose that
$\nabla\psi|_V$ is a diffeomorphism onto
the open subset
$O:=\nabla\psi(V)$. Set
\begin{equation}\label{eq:legendre-characterization-algebraic}
 \bm{F}:=\nabla\psi|_V,\qquad
 h(\bm{x}):=\bm{F}^{-1}(\bm{x})\cdot\bm{x}
              -\psi(\bm{F}^{-1}(\bm{x}))
 \quad(\bm{x}\in O).
\end{equation}
The function $h$ is the local Legendre transform of $\psi$ on $V$.
For $\bm{x}=(x_1,x_2)\in O$, set
\[
 \begin{aligned}
 \bm{\nu}_h(\bm{x})
 &:=\frac{(2h_{x_1},2h_{x_2},|\nabla_{\bm{x}}h|^2-1)}
              {1+|\nabla_{\bm{x}}h|^2},\\[1mm]
 \bm{P}_h(\bm{x})
 &:=(x_1,x_2,h(\bm{x}))-h(\bm{x})\bm{\nu}_h(\bm{x}).
 \end{aligned}
\]
Since $\bm{F}$ is a diffeomorphism, its Jacobian matrix
$\operatorname{Hess}_{p_1,p_2}\psi$ is invertible on $V$.

\begin{Proposition}
\label{prop:support-ribaucour}
The maps $\bm{F}$ and $h$ defined in
\eqref{eq:legendre-characterization-algebraic} satisfy
\begin{equation}\label{eq:legendre-characterization-geometric}
 \bm{\nu}_h\circ\bm{F}=\bm{\nu}|_V,\qquad
 \bm{P}_h\circ\bm{F}=\bm{X}_\rho|_V.
\end{equation}
Moreover, the identities
\eqref{eq:legendre-characterization-geometric}
uniquely determine $\bm{F}$ and $h$ by
\eqref{eq:legendre-characterization-algebraic}.
\end{Proposition}

\begin{Proposition}
\label{prop:ribaucour-eigenlines}
Let $\bm{F}$ and $h$ be as in
\eqref{eq:legendre-characterization-algebraic}.
At points where $\operatorname{Hess}_{p_1,p_2}\psi$ is not scalar, $\bm{F}$ maps
the Hessian eigenline fields of $\psi$ to
the Hessian eigenline fields of $h$.
In particular, $\bm{F}$ maps the corresponding integral curves
to one another.
Moreover, $\bm{F}$ maps the scalar points of
$\operatorname{Hess}_{p_1,p_2}\psi$ to the scalar points of
$\operatorname{Hess}_{x_1,x_2}h$.
The indices of the corresponding Hessian eigenline fields
at isolated scalar points are equal.

If $\bm{X}_\rho$ is regular on $V$, then the following hold:
\begin{enumerate}
\renewcommand{\labelenumi}{(\alph{enumi})}
\item
At nonumbilic points of $\bm{X}_\rho$ in $V$,
the principal line fields, pulled back to $V$ by $\bm{X}_\rho$,
are the Hessian eigenline fields of $\psi$.
\item
A point in $V$ is an umbilic of $\bm{X}_\rho$ if and only if
$\operatorname{Hess}_{p_1,p_2}\psi$ is scalar there.
\item
The index of a principal line field at an isolated umbilic
of $\bm{X}_\rho$ in $V$ is equal to the index of a Hessian eigenline field of $\psi$
at the corresponding scalar point of $\operatorname{Hess}_{p_1,p_2}\psi$.
\end{enumerate}
\end{Proposition}

If $(\operatorname{Hess}_{p_1,p_2}\psi)(\bm{p}_0)$ is invertible, the inverse function theorem
gives a neighborhood $V$ of $\bm{p}_0$ on which
$\nabla\psi$ is a diffeomorphism onto its image.
The local Legendre transform
$h$ of $\psi$ gives the Ribaucour representation $\bm{P}_h$ of $\bm{X}_\rho$
with parameters $\bm{x}=\nabla\psi(\bm{p})$.

\begin{proof}[Proof of Proposition~\ref{prop:support-ribaucour}]
\emph{Step 1:}
We first prove the uniqueness of $\bm{F}$ and $h$.
Assume only that $\bm{F}:V\to O$ is a diffeomorphism and that
$h\in C^\infty(O)$ satisfies
\eqref{eq:legendre-characterization-geometric}.
Set
$(a,b):=\nabla h(\bm{F}(\bm{p}))$.
By the definitions of $\bm{\nu}_h$ and $\bm{\nu}$, the equality
$\bm{\nu}_h(\bm{F}(\bm{p}))=\bm{\nu}(\bm{p})$ gives
\[
 \frac{(2a,2b,a^2+b^2-1)}{1+a^2+b^2}
 =
 \frac{(2p_1,2p_2,p_1^2+p_2^2-1)}{1+p_1^2+p_2^2}.
\]
So we have $(p_1,p_2)=(a,b)$, that is,
\begin{equation}\label{eq:legendre-characterization-gradient}
 \nabla h(\bm{F}(\bm{p}))=\bm{p}.
\end{equation}
Since $|\bm{\nu}|=1$ and $\operatorname{grad}_{S^2}\rho$
is tangent to $S^2$, \eqref{eq:appendix-support-parametrization} gives
\[
 \bm{X}_\rho(\bm{p})\cdot\bm{\nu}(\bm{p})
 =\rho(\bm{p}).
\]
Using the assumed identities in
\eqref{eq:legendre-characterization-geometric}
and the definitions of $\bm{P}_h$ and $\bm{\nu}_h$, we obtain
\[
 \begin{aligned}
 \rho(\bm{p})
 &=\bm{P}_h(\bm{F}(\bm{p}))\cdot\bm{\nu}_h(\bm{F}(\bm{p}))\\
 &=\frac{2\bm{p}\cdot\bm{F}(\bm{p})
       +( |\bm{p}|^2-1)h(\bm{F}(\bm{p}))}{D}
       -h(\bm{F}(\bm{p}))\\
 &=\frac{2\bigl(\bm{p}\cdot\bm{F}(\bm{p})
       -h(\bm{F}(\bm{p}))\bigr)}{D}.
 \end{aligned}
\]
Since $\psi=\rho D/2$, it follows that
\begin{equation}\label{eq:legendre-characterization-potential}
 \psi(\bm{p})
 =\bm{p}\cdot\bm{F}(\bm{p})-h(\bm{F}(\bm{p})).
\end{equation}
Using
\eqref{eq:legendre-characterization-gradient}, we have
\[
 d\psi
 =\bm{F}\cdot d\bm{p}+\bm{p}\cdot d\bm{F}
       -\nabla h(\bm{F})\cdot d\bm{F}
 =\bm{F}\cdot d\bm{p}.
\]
Thus $\bm{F}=\nabla\psi|_V$. Substituting $\bm{x}=\bm{F}(\bm{p})$ into
\eqref{eq:legendre-characterization-potential},
we obtain the formula for $h$ in
\eqref{eq:legendre-characterization-algebraic}.

\smallskip
\noindent\emph{Step 2:}
We next verify
\eqref{eq:legendre-characterization-geometric}
for $\bm{F}$ and $h$ defined by
\eqref{eq:legendre-characterization-algebraic}.
Since $h$ is the local Legendre transform of $\psi$ on $V$,
\eqref{eq:legendre-duality-identities} gives
\eqref{eq:legendre-characterization-gradient}.
The definition of $\bm{\nu}_h$ now gives
$
 \bm{\nu}_h(\bm{F}(\bm{p}))=\bm{\nu}(\bm{p}).
$

We next show the second relation in
\eqref{eq:legendre-characterization-geometric}. Set
$\bm{Q}:=\bm{P}_h\circ\bm{F}$.
In the following, we write $\bm{x}=\bm{F}(\bm{p})$ and
$$
\hat h:=h\circ\bm F.
$$
By \eqref{eq:legendre-characterization-algebraic},
$\hat h=\bm{p}\cdot\bm{x}-\psi$.
Since $\bm{Q}=(x_1,x_2,\hat h)-\hat h\bm{\nu}$ and $|\bm{\nu}|=1$,
we obtain
\[
 \bm{Q}\cdot\bm{\nu}
 =\frac{2\bm{p}\cdot\bm{x}
       +( |\bm{p}|^2-1)\hat h}{D}
       -\hat h
 =\frac{2(\bm{p}\cdot\bm{x}
       -\hat h)}{D}
 =\frac{2\psi}{D}=\rho.
\]
By \eqref{eq:legendre-characterization-gradient},
$d\hat h
=\nabla_{\bm{x}}h(\bm{F}(\bm{p}))\cdot d\bm{x}
=\bm{p}\cdot d\bm{x}$.
Together with $\bm{\nu}\cdot d\bm{\nu}=0$, this gives
\[
 \bm{\nu}\cdot d\bm{Q}
 =\frac{2\bm{p}\cdot d\bm{x}
       +( |\bm{p}|^2-1)d\hat h}{D}
       -d\hat h
 =\frac{(|\bm{p}|^2+1)d\hat h}{D}
       -d\hat h=0.
\]
We show that these identities determine $\bm{Q}$. We can write
\[
 \bm{Q}=\rho\bm{\nu}+\bm{Q}_T,\qquad
 \bm{Q}_T\cdot\bm{\nu}=0.
\]
Differentiating $\bm{Q}\cdot\bm{\nu}=\rho$ and using
$\bm{\nu}\cdot d\bm{Q}=0$, we find
$
 d\rho=\bm{Q}\cdot d\bm{\nu}
       =\bm{Q}_T\cdot d\bm{\nu},
$
which yields
\[
 \bm{Q}_T(\bm{p})
 =(\operatorname{grad}_{S^2}\rho)(\bm{\nu}(\bm{p})).
\]
Hence $\bm{Q}=\bm{X}_\rho$, that is,
\begin{equation}\label{eq:ribaucour-support-identification}
 \bm{P}_h\circ\bm{F}=\bm{X}_\rho|_V.
\end{equation}
This proves \eqref{eq:legendre-characterization-geometric}.
\end{proof}

\begin{proof}[Proof of Proposition~\ref{prop:ribaucour-eigenlines}]
Applying \eqref{eq:legendre-duality-identities3} with
$\bm{x}=\bm{F}(\bm{p})$ gives
\begin{equation}\label{eq:ribaucour-inverse-hessian}
 (\operatorname{Hess}_{x_1,x_2}h)(\bm{F}(\bm{p}))=
 \bigl((\operatorname{Hess}_{p_1,p_2}\psi)(\bm{p})\bigr)^{-1}.
\end{equation}
Since $\bm{F}=\nabla\psi$, the Jacobian matrix of $\bm{F}$ is
$\operatorname{Hess}_{p_1,p_2}\psi$.
By \eqref{eq:ribaucour-inverse-hessian}, the Hessian of $h$
at $\bm{F}(\bm{p})$ is scalar exactly when the Hessian of $\psi$ at $\bm{p}$
is scalar. At other points, the Hessian of $\psi$ and its inverse
have the same eigenspaces in the respective coordinate frames.
Since $d\bm{F}$ acts by the Hessian of $\psi$, it sends each eigenline to
the corresponding eigenline of the Hessian of $h$.
Thus $\bm{F}$ sends the corresponding integral curves to one another.
The equality of indices follows from the invariance of
the line-field index under a diffeomorphism.

Assume that $\bm{X}_\rho$ is regular on $V$.
By Proposition~\ref{prop:support-ribaucour}, we have
$\bm{P}_h\circ\bm{F}=\bm{X}_\rho|_V$.
The classical Ribaucour correspondence
\cite[Appendix~A, Fact~A.1]{AFU} identifies the principal line fields
and the umbilics of $\bm{P}_h$ with the Hessian eigenline fields
and the points where the Hessian of $h$ is scalar, respectively.
The correspondence for $\psi$ in \textup{(a)} and \textup{(b)} then follows
from \eqref{eq:ribaucour-inverse-hessian} and the eigenline correspondence.
The invariance of the line-field index under a diffeomorphism
gives the equality of indices in \textup{(c)}.
\end{proof}


\begin{thebibliography}{99}

\bibitem{AFU}
N.~Ando, T.~Fujiyama and M.~Umehara,
\newblock $C^1$-umbilics with arbitrarily high indices,
\newblock \emph{Pacific J. Math.} \textbf{288} (2017), 
1--26.

\bibitem{AU2023}
N.~Ando and M.~Umehara,
\newblock Umbilics of surfaces in the Lorentz--Minkowski $3$-space,
\newblock \emph{Results Math.} \textbf{78} (2023), Paper No.~241.
\newblock doi:~\href{https://doi.org/10.1007/s00025-023-02013-2}{10.1007/s00025-023-02013-2}.

\bibitem{Alpoge2026}
L.~Alp\"oge,
\newblock Announcement of counterexamples to the Loewner and
Carath\'eodory conjectures in the smooth category,
\newblock X post, 19 August 2026.
\newblock \url{https://x.com/__alpoge__/status/2089971359921156203}.

\bibitem{GS1997}
C.~Guti\'errez and F.~S\'anchez-Bringas,
\newblock Planar vector field versions of Carath\'eodory's and
Loewner's conjectures,
\newblock \emph{Publ. Mat.} \textbf{41} (1997), 
169--179.

\end{thebibliography}
\end{document}